\documentclass[a4pape, reqno,12pt,oneside]{amsart}

\usepackage{latexsym}
\usepackage{amssymb}
\usepackage{graphicx}
\usepackage{wrapfig} 
\usepackage{amsthm}
\usepackage{float}
\usepackage{epsfig}
\usepackage{multirow}
\usepackage[toc,page]{appendix}
\usepackage[footnotesize]{caption}
\RequirePackage{color}
\usepackage{cite}
\usepackage{extarrows} 
\usepackage{tikz}
\usetikzlibrary{arrows,snakes,backgrounds,trees}
\usetikzlibrary{matrix,decorations.pathreplacing,positioning}
\usetikzlibrary{decorations.markings}

\tikzstyle directed=[postaction={decorate,decoration={markings, mark=at position .65 with {\arrow{stealth}}}}]

\numberwithin{figure}{section}
\usepackage{subfigure}

\usepackage{amscd,enumerate}
\usepackage{amsfonts}
\usepackage{enumitem} 

\input xy
\xyoption{all}

\usepackage{multicol}
\usepackage{hyperref}
\hypersetup{ colorlinks,
linkcolor=blue,
filecolor=green,
urlcolor=blue,
citecolor=blue }

\newtheorem{lemma}{Lemma}[section]
\newtheorem{corollary}[lemma]{Corollary}
\newtheorem{theorem}[lemma]{Theorem}
\newtheorem{proposition}[lemma]{Proposition}

\newtheorem{definition}[lemma]{Definition}

\newtheorem{example}[lemma]{Example}

\usepackage{tikz}
\usetikzlibrary{arrows,snakes,backgrounds,trees}
\usepackage{stmaryrd}

\def\T+{{\mathbb T_d^+}}

\def\A{\mathcal{A}}
\def\dist{\mathop{\rm dist}}

\def\N{{\mathcal N}}

\def\SRW {\mathop {\rm SRW}\nolimits}
\def\RW {\mathop {\rm RW}\nolimits}

\makeatletter
\@namedef{subjclassname@2020}{%
  \textup{2020} Mathematics Subject Classification}
\makeatother

\begin{document}


\title{On Hilbert evolution algebras of a graph}

\author[Sebastian J. Vidal]{Sebastian J. Vidal}
\address{Sebastian J. Vidal: Departamento de Matem\'atica, Facultad de Ingenier\'ia, Universidad Nacional de la Patagonia ``San Juan Bosco'', Km 4, CP 9000, Comodoro Rivadavia, Chubut, Argentina.}
\email{sebastianvidal79@gmail.com}

\author[Paula Cadavid]{Paula Cadavid}
\address{Paula Cadavid: Universidade Federal do ABC, Avenida dos Estados, 5001-Bangu-Santo Andr\'e - SP, Brazil.}
\email{pacadavid@gmail.com}

\author[Pablo M. Rodriguez ]{Pablo M. Rodriguez}
\address{Pablo M. Rodriguez: Centro de Ci\^encias Exatas e da Natureza, Universidade Federal de Pernambuco, Av. Prof. Moraes Rego, 1235 - Cidade Universit\'aria - Recife - PE, Brazil.}
\email{pablo@de.ufpe.br}

\begin{abstract}
Evolution algebras are a special class of non-associative algebras exhibiting connections with different fields of Mathematics. Hilbert evolution algebras generalize the concept through a framework of Hilbert spaces. This allows to deal with a wide class of infinite-dimensional spaces. In this work we study Hilbert evolution algebras associated to a graph. Inspired in definitions of evolution algebras we define the Hilbert evolution algebra associated to a given graph and the Hilbert evolution algebra associated to the symmetric random walk on a graph. For a given graph, we provide conditions under which these structures are or are not isomorphic. Our definitions and results extend to graphs with infinitely many vertices a similar theory developed for evolution algebras associated to finite graphs.
\end{abstract}

\maketitle


\section{Introduction}

\subsection{Evolution algebras}

Evolution algebras are a special class of non-associative algebras  exhibiting connections with different fields of Mathematics. The first reference of a Theory of Evolution Algebras is due to Tian and Vojtechovsky, see \cite{tv}, who define an evolution algebra as an algebra $\A:=(\A,\cdot\,)$ over a field $\mathbb{K}$ with a basis $S:=\{e_i\}_{i \in \Lambda}$ such that

\begin{equation}\label{eq:ea}
e_i \cdot e_i =\displaystyle \sum_{k \in \Lambda} c_{ik} e_k, \text{ for any }i \text{ and } e_i \cdot e_j =0,\text{ if } i\neq j.
\end{equation} 

\smallskip\noindent
The concept was motivated by evolution laws of genetics: if one think in alleles as generators of algebras, then reproduction in genetics is represented by multiplication in algebra. Indeed, the first properties stated by \cite{tv} are of interest from this biological point of view. The theory was developed further on in the seminal work \cite{tian}, where the author show many interesting correspondences between these algebras and subjects like discrete-time Markov chains, graph theory, group theory, and others. We refer the reader to \cite{Cabrera/Siles/Velasco,Cabrera/Siles/Velasco/2,PMP,PMP2,PMPT,PMPY,camacho/gomez/omirov/turdibaev/2013,COT,cardoso,casas/ladra/omirov/rozitov/2013,casas/ladra/rozitov/2011,Elduque/Labra/2015,Elduque/Labra/2019,falcon,paniello,reis,tian2} and references therein for an overview of recent theoretical results, applications to other fields, and interesting open problems. We point out that, such defined, these algebras are a special class of the genetic algebras introduced by \cite{Schafer}.

The connections with other fields, suggested in part of the literature mentioned above, works mainly for finite-dimensional spaces. For applications involving infinite-dimensional spaces still there exist some gaps to be addressed, as pointed out recently by \cite{SPP}. The reason is that the previous definition considers a countable basis and it implicitly assumes that it is a Hamel basis. Thus, the sum in \eqref{eq:ea} can have only a finite number of nonzero terms. Such a restriction has implications, for example, in the connection with Markov chains. Indeed, in \cite[Example 1.2]{SPP} the authors show that not all Markov chain with a countably infinite state space has associated an evolution algebra, in contrast with \cite[Theorem  16,  page  54]{tian} for some Markov chains. In order to avoid that restriction in the series, \cite{SPP} introduce a generalization of evolution algebra by evoking an approach involving Hilbert spaces. Such approach leads us to consider other kind of basis; namely, Schauder basis. This gave rise to the concept of Hilbert evolution algebra which is an extension obtained by providing an evolution algebra structure in a given Hilbert space. 

\subsection{Hilbert evolution algebras}

Let  $V$ be  a real or complex  vector space with  inner product $ \langle \cdot, \cdot \rangle$ which is finite or infinite-dimensional. A subset $\{e_j\}_{j\in\Lambda}\subset V$, where $\Lambda$  is a countable set,  is a Schauder Basis of $V$ if any $v\in V$ has an unique representation
	\begin{equation} \label{rep} \nonumber
	v=\sum_{j\in\Lambda} v_j e_j, \,\,\text{  where } v_{j}  \in  \mathbb{K}.
	\end{equation}
We say that $V$ is a Hilbert space if it is also a complete metric space with respect to the distance function induced by the inner product.  A subset $\{e_j\}_{j\in\Lambda}\subset V$ is an orthonormal basis if every $v\in V$ can be expressed as
\begin{equation} \nonumber
    v=\sum_{j\in\Lambda}\langle v,e_j\rangle e_j.
\end{equation}
On the other hand, we  say that $V$ is separable if it has a countable dense subset. In this case, any orthonormal basis is countable. The Gram-Schmidt orthonormalization process proves that every separable Hilbert space has an orthonormal basis. We highlight  that if $V$ is finite-dimensional, the notion of Schauder basis coincides with that of Hamel basis.

\begin{definition}\label{def:evolalg} Let $\A=(\A,\langle\cdot, \cdot\rangle)$ be a real or complex separable Hilbert space provided with an algebra structure by the product $\cdot:\A\times\A\rightarrow \A$. We say that $\A:=(\A,\langle\cdot, \cdot\rangle,\cdot \,)$ is a separable Hilbert evolution algebra if it satisfies:
\begin{enumerate}[label=(\roman*)]
\item There exists an orthonormal basis $\{e_i\}_{i\in\mathbb{N}}$ and scalars $\{c_{ki}\}_{i,k\in\mathbb{N}}$, such that 
\begin{equation}\label{eq:ea03}
e_i \cdot e_i =\displaystyle \sum_{k=1}^{\infty} c_{ki} e_k 
\end{equation}
and
\begin{equation}\label{eq:ea04}
e_i \cdot e_j=0, \text{ if }i\neq j,
\end{equation} 
for any $i,j\in \mathbb{N}$. 
\smallskip
\item For any $v\in\A$, the left multiplications $L_v: \A  \longrightarrow  \A$ defined by $L_{v}(w):= v\cdot w$ for any $w\in \A$ are continuous in the metric topology induced by the inner product; i.e., there exists constants $M_v>0$ such that
\begin{equation}\label{eq:continuity condition}
\|L_v(w)\|\leq M_v \|w\|,\text{ for all } w\in \A.
\end{equation}
\end{enumerate}
\end{definition}

Conditions (i) and (ii) are compatibility conditions between the structures. A basis satisfying condition (i) will be called {\it orthonormal natural basis}.  In the sequel we will work only with separable Hilbert spaces but, for the sake of simplicity, we omit the word separable and talk about Hilbert evolution algebras. It is useful to write the product explicitly in terms of the base $\{e_i\}_{i\in\mathbb{N}}$. Let $v, w\in\A$ and write $v=\sum_{k=1}^{\infty}v_k e_k$, $w=\sum_{k=1}^{\infty}w_k e_k$, then extending \eqref{eq:ea03} by linearity, we have
\begin{equation*}
v\cdot w=\sum_{k=1}^{\infty}\biggl(\sum_{i=1}^{\infty} v_iw_ic_{ki}\biggr) e_k.
\end{equation*}




\begin{proposition} \cite[Corollary 2.4 ]{SPP}\label{corol:manageable condition}
Let $\A$ be a separable Hilbert space. Consider an orthonormal basis $\{e_i\}_{i\in\mathbb{N}}$ and suppose that the scalars $\{c_{ki}\}_{i,k\in\mathbb{N}}$ satisfy
\begin{equation}\label{eq:condition HEA02}
K:=\sup\biggl\{\sum_{k=1}^{\infty}|c_{ki}|^2 \,:\, i\in\mathbb{N} \biggr\}<+\infty.
\end{equation}
Then $\A$ admits an Hilbert evolution algebra structure satisfying equations \eqref{eq:ea03} and \eqref{eq:ea04}.
\end{proposition}

\begin{example}\label{exe:mc} Let $\{p_{ik}\}_{i,k\in\mathbb{N}}$ be the transition probabilities of a discrete-time Markov chain with states space $\mathbb{N}$; that is, $p_{ik}\in[0,1]$ for any $i,k\in\mathbb{N}$ and
$$\sum_{k=1}^{\infty}p_{ik}=1,$$
for all $i\in \mathbb{N}$. It is not difficult to see that the sequence of scalars $\{c_{ki}\}_{i,k\in\mathbb{N}}$ such that $c_{ki}=p_{ik}$ satisfy \eqref{eq:condition HEA02}. Indeed, 
$$\sum_{k=1}^{\infty}|c_{ki}|^2 = \sum_{k=1}^{\infty}|p_{ik}|^2\leq \sum_{k=1}^{\infty}p_{ik}=1.$$
In words, contrary to what happens in evolution algebras, here we have a way of associating to any discrete-time Markov chain a Hilbert evolution algebra. For more details on this connection we refer the reader to \cite{SPP}. 
\end{example}

Let $\A$ be a Hilbert evolution algebra with a natural basis $\{e_i\}_{i \in \mathbb{N}}$ and structural constants $\{c_{ki}\}_{i,k \in \mathbb{N}}$. We define the evolution operator as the linear operator $C:D(C)\longrightarrow \A$ given by its values in a natural orthonormal basis,
\begin{equation*}
C(e_i):=e_i^2=\sum_{k=1}^{\infty} c_{ki}e_k ,
\end{equation*}
and with domain $D(C)\subset \A$ defined as
\begin{equation}\label{eq:domain of C}
D(C):=\Bigg\{v=\sum_{i=1}^{\infty} v_ie_i\in\A :\, 
\sum_{k=1}^{\infty} \biggl|\sum_{i=1}^{\infty} v_ic_{ki}\biggr|^2<\infty \Bigg\}.
\end{equation}
With this we can write
\begin{equation}\label{eq:EOp00}
	C(v):=\sum_{k=1}^{\infty} \biggl(\sum_{i=1}^{\infty} v_ic_{ki}\biggr)e_k,
\end{equation}
for any $v=\sum_{i=1}^{\infty} v_i e_i \in D(C)$.
In the general case the operator $C$ will be unbounded, thus it is important to find conditions on the structure constants to know when $C$ is closable, closed or bounded. 
	
\begin{proposition}\cite[Proposition 2.5]{SPP}\label{prop:EOcont}
Let $\A$ be a Hilbert evolution algebra with structure constants satisfying one of the following conditions:
\begin{enumerate}[label=(\roman*)]
    \item \label{item:condition continuity C}
    $\displaystyle\sum_{k=1}^{\infty}\sum_{i=1}^{\infty}|c_{ki}|^2<\infty.$ \vspace{0.2cm}
   	
    \item (Schur Test) There exists $\alpha_k, \beta_i>0$, $i,k\in \mathbb{N}$ and $M_1, M_2>0$ such that
    \begin{equation}\label{eq:schur test}
    \begin{array}{l}
    \displaystyle\sum_{k=1}^{\infty}|c_{ki}|\alpha_k\leq M_1\beta_i, \quad\text{ for all } i\in \mathbb{N}, \\[1.1ex]
    \displaystyle\sum_{i=1}^{\infty}|c_{ki}|\beta_i\leq M_2\alpha_k,
    \quad\text{ for all } k\in \mathbb{N}.
    \end{array}
    \end{equation}
\end{enumerate}
Then $D(C)=\A$ and the evolution operator $C:\A\longrightarrow \A$ is bounded with
$\|C\|\leq (M_1M_2)^{1/2}.$
\end{proposition}

\subsection{Our contribution and outline of the paper}
The purpose of our work is to advance in the study of Hilbert evolution algebras by extending results coming from the interplay between evolution algebras and graphs. Such connection has been a subject of current research, see for instance \cite{PMP,PMP2,Elduque/Labra/2015, Elduque/Labra/2019}, and references therein. One of the open problems related to the subject is to understand the relationship between the evolution algebra determined by a graph and the evolution algebra induced by the symmetric random walk on the same graph. This question has been stated by \cite{tian,tian2} motivated in the formulation of a new landscape in discrete geometry. Partial results addressing this question have been presented recently by \cite{PMP,PMPT}, where the focus has been the search of conditions on the graph to guarantee the existence of isomorphisms between these algebras. 

Here we introduce the analogous definitions; namely, the Hilbert evolution algebra associated to a graph and the Hilbert evolution algebra associated to the symmetric random walk on a graph. Then we explore their connection through the analysis of the homomorphisms connecting them. We provide conditions under which these structures are isomorphic, in a sense to be defined later but covering the finite-dimensional case, and we extend our discussion to the existence of unitary isomorphisms. As a sideline, we point out that understanding the behavior of these structures related to an infinite graph can leave us to the development of new techniques to deal with general Hilbert evolution algebras by mean of tools of Graph Theory. 

The paper is organized as follows. Section 2 is devoted to a background of Graph Theory with preliminary results which allow to state our definitions of Hilbert evolution algebras associated to a graph. In Section 3 we formalize the concept of isomorphism between these new structures and we state and prove our main theorem. Our discussion is extended to the case of unitary isomorphisms, and the respective results are gathered in Section 4. Finally, we dedicate Section 5 to a discussion of open problems and suggestions for further research.

\section{Hilbert evolution algebras associated to a graph}\label{sec:Hilbert evolution algebras associated to a graph}



\subsection{Brief Background of Graph Theory}

Although we assume that the reader is familiar with the basic notation of Graph Theory, see for example \cite{godsil}, we review some of it in the sequel for the sake of clarity. Let $G=(V,E)$ be a locally finite graph; that is, a graph such that  $\deg(i)<\infty$ for all $i\in V$, where $\deg(i)$ denotes the degree of vertex $i$.  We say that $G$ has uniformly bounded degree if there exists  $M>0$ such that $\deg (i)\leq M,$ for all $i\in V$. Evidently if $G$ has uniformly bounded degree then $G$ is locally finite. The adjacency matrix of $G$ is denoted by $A:=A(G)=(a_{ij})$, where $a_{ij}$ is the number of edges between the vertices $i$ and $j$. Here we consider simple graphs, that is graphs without multiple edges or loops, so $a_{ii}=0$ and  $a_{ij}\in\{0,1\}$ for any $i,j\in V$. 

Our first task is to give a definition of Hilbert evolution algebra associated to a given graph. In order to do it, let us consider  a locally finite  graph $G=(V,E)$ such that $V=\mathbb{N}$ and let us add some additional structure to these objects. We shall see later that our approach trivially holds for finite graphs. Following \cite{Mohar, Mohar Woess}, we will use the $\ell^2(\mathbb{N})$ Hilbert space of sequences $\{v_i\}_{i\in\mathbb{N}}$ of scalars such that $\sum_{i=1}^{\infty}|v_i|^2<\infty$, and we denote by $\delta_j:=\{\delta_{ij}\}_{i\in\mathbb{N}}$ the standard orthonormal basis of $\ell^2(\mathbb{N})$. Remember that the standard inner product in $\ell^2(\mathbb{N})$ is given by 
\begin{equation} \nonumber
\langle v,w\rangle :=\sum_{i=1}^{\infty} v_i \overline{w}_i,
\end{equation}
where $v=\sum_{i=1}^{\infty}v_i\delta_i$ and $w=\sum_{i=1}^{\infty}w_i\delta_i$. In this framework, the adjacency matrix of $G$ can be interpreted as an operator $A$ densely defined in $\ell^2(\mathbb{N})$. That is, if $D_0$ is the dense subset of $\ell^2(\mathbb{N})$ formed by those sequences with only a finite number of non-zero components, we can define the operator $A_0:\ell^2(\mathbb{N})\longrightarrow \ell^2(\mathbb{N})$ by
\begin{equation*}
A_0(\delta_i):=\sum_{k=1}^{\infty}a_{ki} \delta_k
\end{equation*} 
and we extend it by linearity to all $D_0$, giving the domain $D(A_0)=D_0$. Remember that $a_{ki}$ equals $1$ or $0$ according whether there is, or not, an edge between vertices $k$ and $i$, respectively. This operator is well defined because $G$ is locally finite; that is, $A_0(\delta_i)\in \ell^2(\mathbb{N})$ which follows from
\begin{equation*}
\sum_{k=1}^{\infty}|a_{ki}|^2= \sum_{k=1}^{\infty}a_{ki}=\deg(i)<\infty.
\end{equation*}
The domain $D_0$ is dense in $\ell^2(\mathbb{N})$ so the adjoint $A_0^*:\ell^2(\mathbb{N})\longrightarrow\ell^2(\mathbb{N})$ is well defined. Note that the operator $A_0$ is symmetric; i.e.,
\begin{equation*}
\langle A_0 (\delta_i),\delta_k\rangle=\langle \delta_i,A_0(\delta_k)\rangle,
\quad \text{ for } i,j\in \mathbb{N},
\end{equation*}
because $a_{ik}=a_{ki}$.
Now we recall that a densely defined symmetric operator $T$ is always closable, because the adjoint $T^*$ is a closed linear operator and  $T^*$ is an extension of $T$ (see \cite[pp.38]{Schmudgen}). Thus $A_0$ is closable and their closure will be called the adjacency operator of $G$ and denoted by
\begin{equation*}\label{eq:adjacency operator01}
A:=\overline{A}_0:D(A)\longrightarrow \ell^2(\mathbb{N})
\end{equation*}
where the domain of $A$ is 
\begin{equation}\label{eq:domain of A}
D(A):=\biggl \{v=\sum_{i=1}^{\infty}v_i\delta_i\in \ell^2(\mathbb{N})\,:\, \sum_{k=1}^{\infty}\biggl|\sum_{i=1}^{\infty}
v_i a_{ki} \biggr|^2 <\infty \biggr\}.
\end{equation}
Hence, the operator is given by
\begin{equation}\label{eq:adjacency operator02}
A(v)=\sum_{k=1}^{\infty}\biggl(\sum_{i=1}^{\infty} v_ia_{ki} \biggr)\delta_k,
\end{equation}
for any $v= \sum_{i=1}^{\infty} v_i\delta_i \in D(A)$. Now we must analyze the properties of the domain of $A$ and the possibility to extend the operator to the entire space $\ell^2(\mathbb{N})$.  This analysis is simplified because $A$ is symmetric. We use that the adjacency operator is real to obtain that there exists a self-adjoint extension of $A$, which in general is not unique. 
Since this extension is unique if, and only if,  $A$ is self-adjoint \cite{Mohar,Schmudgen}, we focus our attention in the case of self-adjointness.
\begin{theorem}\label{theo:adjacency oper bounded} 
Let $G$ be a locally finite graph. Then the  adjacency operator $A:\ell^2(\mathbb{N})\longrightarrow \ell^2(\mathbb{N})$ is bounded 
if, and only if, $G$ has uniformly bounded degree. Specifically, if $M>0$ is such that 
\begin{equation}\label{eq:ubdegree} 
\deg (i)\leq M,\quad\text{ for all } i\in \mathbb{N},
\end{equation}
then $\|A\|\leq M$. 
Moreover,  $A$ is self-adjoint.
\end{theorem}
\begin{proof}
The first statements of the Theorem were proved by Mohar \cite[Theorem 3.2]{Mohar}. 
Let us prove that $A$ is self-adjoint. We know that $A$ is continuous, thus there exists its adjoint operator $A^*:\ell^2(\mathbb{N})\longrightarrow \ell^2(\mathbb{N})$ defined by
\begin{equation}\nonumber
\langle A^*v,w \rangle=\langle v,Aw \rangle,
\end{equation}
for all $v,w\in \ell^2(\mathbb{N})$. But $A$ is symmetric, so it follows that $A=A^*$.
\end{proof}

\subsection{Hilbert Evolution Algebra associated to a graph}
The considerations of the previous subsection allow us to carry out the following analysis.

\begin{proposition}\label{prop:HEA of uniformily bounded degree}
Let $G$ be a graph with adjacency operator $A:D(A)\longrightarrow \ell^2(\mathbb{N})$ given by equation \eqref{eq:adjacency operator02} and $\A$ be a separable Hilbert space with an orthonormal basis $\{e_i\}_{i\in\mathbb{N}}$.
Then, the algebra $\A(G)$ defined by 
\begin{equation*}\label{eq:Gea01}
e_i \cdot e_i =\displaystyle \sum_{k=1}^{\infty} a_{ki} e_k \text{ and } e_i \cdot e_j=0, \text{ if }i\neq j,
\end{equation*}
is a Hilbert evolution algebra if and only if $G$ has uniformly bounded degree. This $\A(G)$ will be called the Hilbert evolution algebra associated to $G$.
\end{proposition}

\begin{proof}
If $\A(G)$ is a Hilbert evolution algebra, there exists a constant $K>0$ such that $\|L_v w\|\leq K\|w\|$, for all $v,w\in\A(G)$. In particular $\|L_{e_i} e_i\|\leq K$, for all $i\in\mathbb{N}$, that is
\begin{equation*}
\deg(i)=\sum_{k=1}^\infty a_{ki}=\sum_{k=1}^\infty |a_{ki}|^2=\|L_{e_i} e_i\|^2\leq K^2,
\end{equation*}
for all $i\in\mathbb{N}$. That is, $G$ is of uniformly bounded degree.  
Reciprocally, let $M>0$ such that $\deg(i)\leq M$ for all $i\in\mathbb{N}$. Then,
\begin{equation*}
\sum_{i=1}^\infty |a_{ki}|^2 =\sum_{i=1}^\infty a_{ki} =\deg(i)\leq M,\quad\text{for all } i\in\mathbb{N}.
\end{equation*}
Thus, by Proposition \ref{corol:manageable condition} we can conclude that $\A(G)$ is a Hilbert evolution algebra.

\end{proof}


It is worth pointing out that our results shows the strong relation between the concepts. On one hand we see that the adjacency operator is bounded if, and only if, it is of uniformly bounded degree (see Theorem \ref{theo:adjacency oper bounded}). On the other hand, Proposition \ref{prop:HEA of uniformily bounded degree} says that these are the only class of graphs which induces a Hilbert evolution algebra $\A(G)$. Next, we see that in this framework, the evolution operator is well behaved.

\begin{proposition}\label{prop:ev_op_cont}
If $G$ has uniformly bounded degree, the operator $C:\A(G)\longrightarrow\A(G)$ is bounded.
\end{proposition}
\begin{proof}
Recall that the structure constants are given by the adjacency operator, i.e., $c_{ij}=a_{ij}$. Again, the symmetry of $A$ implies 
\begin{equation}\nonumber
\deg(i)=\sum_{k=1}^{\infty} a_{ki}=\sum_{k=1}^{\infty} a_{ik}\leq M, \quad\text{ for all } i\in \mathbb{N}.
\end{equation}
Thus we can apply the Schur test \eqref{eq:schur test} of Proposition \ref{prop:EOcont} to conclude that $C$ is bounded.
\end{proof}

For the next result we recall the following definition (see \cite{Sunder}). Let $\mathcal{H, K}$ be Hilbert spaces and $T:\mathcal{H}\longrightarrow\mathcal{H}$, $S:\mathcal{K}\longrightarrow\mathcal{K}$ two bounded operators. We say that they are unitarily equivalent if there exists a unitary operator $U:\mathcal{H}\longrightarrow\mathcal{K}$ such that $S=UTU^{*}$. This concept describes when two bounded operators in Hilbert spaces are essentially the same.

\begin{proposition}
Let be a graph with uniformly bounded degree and $\A$ be a separable Hilbert space with an orthonormal basis $\{e_i\}_{i\in\mathbb{N}}$. Then, the adjacency operator $A:\ell^2(\mathbb{N})\longrightarrow \ell^2(\mathbb{N})$ and the evolution operator $C:\A(G)\longrightarrow\A(G)$ 
are unitarily equivalent. In particular this implies that $C$ is self-adjoint.
\end{proposition}

\begin{proof}
First, note that, by the uniformly boundedness  of $G$ we can apply the Theorem \ref{theo:adjacency oper bounded} to obtain that $A:\ell^2(\mathbb{N})\longrightarrow \ell^2(\mathbb{N})$ is well defined, bounded and self-adjoint. 
On the other hand, by Proposition \ref{prop:ev_op_cont}, $C:\A(G)\longrightarrow\A(G)$  is bounded.
Next, note that the spaces $\ell^2(\mathbb{N})$ and $\A$ are both Hilbert separable, hence they are isometrically isomorphic. This implies that there is an unitary operator $U:\ell^2(\mathbb{N})\longrightarrow \A$ taking the orthonormal basis $\{\delta_i\}_{i\in\mathbb{N}}$ into the orthonormal basis $\{e_i\}_{i\in\mathbb{N}}$, that is $U(\delta_i):=e_i$. Thus, by the continuity of the operators and representations \eqref{eq:EOp00} and \eqref{eq:adjacency operator02} we obtain 
\begin{equation}\label{eq:evol vs adja}
C(U(v))=\sum_{k=1}^{\infty}\biggl(\sum_{i=1}^{\infty}v_i a_{ki}  \biggr)e_k= U(A(v)),
\end{equation}
for any $v=\sum_{i=1}^{\infty}v_i\delta_i\in \ell^2(\mathbb{N})$. Hence we can conclude that $C U=UA$, i.e., they are unitarily equivalent, because $U^*=U^{-1}$.
\end{proof}



\subsection{Hilbert Evolution Algebra associated to the symmetric random walk on a graph}

There is a second way to define an evolution algebra associated to a graph 
$G=(V,E)$; it is the one induced by the symmetric random walk ($\SRW$) on $G$. The $\SRW$ is a discrete-time Markov chain $\{X_n\}_{n\geq 0}$ with state space given by $V$ and transition probabilities given by
\begin{equation}\nonumber
\mathbb{P}(X_{n+1}=k|X_{n}=i)=\frac{a_{ki}}{\deg(i)},
\end{equation}
where $i,k,n\in \mathbb{N}$ and $\deg(i)=\sum_{k=1}^\infty a_{ki}.$ Roughly speaking, the sequence of random variables $\{X_n\}_{n\geq 0}$ denotes the set of positions of a particle walking around the vertices of $G$; at each discrete-time step the next position is selected at random from the set of neighbors of the current state. In other words, we can introduce the transition operator of $G$, defined in the basis by
\begin{equation*}
P(\delta_i)=\sum_{k=1}^\infty p_{ik}\delta_k  \end{equation*}
where $p_{ik}:=a_{ki}/\deg(i),$ for $i,k \in \mathbb{N}$. This vector is well defined because 
\begin{equation*}
\|P(\delta_i)\|^2=\sum_{k=1}^\infty \biggr|\frac{a_{ki}}{\deg(i)} \biggl|^2=\frac{\sum_{k=1}^\infty a_{ki}}{\deg(i)^2}=\frac{1}{\deg(i)}\leq 1.
\end{equation*}
Note that we can do a similar analysis to the case of the adjacency operator, that is, analyze the domain of the operator and conditions for boundedness. However, it is not necessary, because, as noticed in Example \ref{exe:mc}, the transition probabilities induce a Hilbert evolution algebra. 

\begin{proposition}\label{prop:SRWhea}
Let $G$ be a graph, $\A$ be a separable Hilbert space and  $\{e_i\}_{i\in\mathbb{N}}$ be an orthonormal basis of $\A$. Then we have a Hilbert evolution algebra associated to a symmetric random walk on $G$, denoted by $\A_{\RW}(G)$ and  defined by the product 
\begin{equation}
\nonumber 
e_i \cdot e_i =\displaystyle \sum_{k=1}^{\infty} \frac{a_{ki}}{\deg(i)} e_k \quad  \text{ and }\quad e_i \cdot e_j=0, \text{ if }i\neq j.
\end{equation}
\end{proposition}
\begin{proof}
As in Example \ref{exe:mc}, we can use the Proposition \ref{corol:manageable condition}, because
\begin{equation*}
\sum_{k=1}^{\infty} \biggr|\frac{a_{ki}}{\deg(i)}\biggl|=1.
\end{equation*}
\end{proof}



Note that in the definition of $\A_{\RW}(G)$ we do not need the condition of the uniform boundedness of the graph. The reason for this are the weights  $1/\deg(i)$, or in other words, the probability condition over the structure constants.

\section{Existence of isomorphisms between $\A_{\RW}(G)$ and $\A(G)$}

\subsection{Definitions and main result}
One of the open questions related to the study of evolution algebras and graphs is the connection between $\A_{\RW}(G)$ and $\A(G)$. This issue has been partially answered by \cite{PMP,PMPT}, where the authors obtain sufficient and necessary conditions under which both objects are isomorphic for a finite non-singular graph. Our purpose is to extended that study to the context of Hilbert evolution algebras associated to a graph with infinite many vertices. Let us start with a preliminary definitions.

\begin{definition}
Let $\A$ and $\A'$ be Hilbert evolution algebras and $f:\A\longrightarrow \A'$ a linear operator. We say that:
\begin{enumerate} [label=(\roman*)]
\item $f$ is a {\it homomorphism of Hilbert evolution algebras} if it is an algebra homomorphism and it is continuous.
\item $f$ is an isomorphism of Hilbert evolution algebras if it is homomorphism of Hilbert evolution, there exist $f^{-1}$  and it is continuous. In this case we say that $\A$ and $\A'$ are isomorphic as Hilbert evolution algebras, and we denote it by $\A\cong \A'$.

\item $f$ is an unitary homomorphism of Hilbert evolution algebras if it is an isometric operator and is an isomorphism of Hilbert evolution algebras.  In this case we say that $\A$ and $\A'$ are unitarily isomorphic as Hilbert evolution algebras,  and we denote it by $\A\cong_U \A'$.

\end{enumerate}
\end{definition}

\begin{definition}
We say that an infinite graph is non-singular if the adjacency operator $A$ satisfies $\ker A=\{0\}$.
\end{definition}

\begin{theorem}\label{theo:criterio}
Let $G$ be a graph with uniformly bounded degree. If $G$ is a regular or a biregular graph, then $\A_{\RW}(G)\cong \A(G)$. Moreover, if $G$ is non-singular and $\A_{\RW}(G)\cong \A(G)$ then $G$ is a regular or a biregular graph. 
\end{theorem}

The proof of Theorem \ref{theo:criterio}, which holds for graphs with infinitely many vertices, is carried out by adapting to our framework the arguments developed in \cite[Theorem 2.3]{PMPT}. We recall that now the notion of isomorphism includes continuity. For the sake of clarity we left the details of the proof for subsection \ref{sec:proofs}.

\begin{corollary}\label{cor:theoCRR}
\cite[Theorem 2.3]{PMPT} Let $G$ be a finite graph. If $G$ is a regular or a biregular graph, then $\A_{\RW}(G)\cong \A(G)$. Reciprocally, if $G$ is a non-singular graph then, $\A_{\RW}(G)\cong \A(G)$ implies that $G$ is a regular or a biregular graph. 
\end{corollary}

\begin{proof}
Note that any finite graph is locally finite and has uniformly bounded degree. Moreover, note that if $G$ is non-singular then $\ker A=\{0\}$. 
\end{proof}

We present Corollary \ref{cor:theoCRR} as stated by \cite[Theorem 2.3]{PMPT}. However, as pointed out for us recently by \cite{henao}, if $G$ is a non-singular graph, then biregularity implies regularity.

\subsection{Preliminary results for the Proof of Theorem \ref{theo:criterio}}

In what follows we shall use regular and biregular graphs. Remember that these are locally finite graphs with uniformly bounded degree so all the definitions introduced in this paper works.

\begin{proposition}\label{theo:generalization}
Let $G$ be a regular or a biregular graph. Then $\A(G) \cong \A_{\RW}(G)$.
\end{proposition}

\begin{proof}Let us prove the proposition for a biregular graph, the regular case is analogous. Assume that $G=(V_1,V_2,E)$  is a $(d_1,d_2)$-biregular  graph and consider, like in \cite[Proposition 2.9]{PMPT}, the linear map $f:\mathcal{A}(G)\longrightarrow \mathcal{A}_{\RW}(G)$ defined by 
\begin{equation}\label{eq:iso}
\nonumber
f(e_i)=\left\{
\begin{array}{cl}
\left(d_1^2 d_2\right)^{1/3}\, e_i,& \text{ if }i\in V_1,\\[.2cm]
\left(d_1 d_2^2\right)^{1/3}\, e_i,& \text{ if }i\in V_2.
\end{array}\right.
\end{equation}
Clearly, $f$ is a linear isomorphism  and behaves well with respect to the algebra product of $\A(G)$. Thus $f$ is an algebra isomorphism. For the continuity of $f$, we write $M_1:=\left(d_1^2 d_2\right)^{1/3}$, $M_2:=\left(d_1 d_2^2\right)^{1/3}$, then
\begin{equation} \nonumber
\nonumber f(v)=\sum_{k=1}^{\infty}v_kf(e_k) =M_1\sum_{k\in V_1} v_ke_k + M_2\sum_{k\in V_2} v_ke_k,
\end{equation}
for $v=\sum_{k=1}^{\infty} v_ke_k \in \A(G)$. Thus
\begin{equation*}
\nonumber \|f(v)\|^2= M_1^2\sum_{k\in V_1} |v_k|^2 + M_2^2\sum_{k\in V_2} |v_k|^2 \leq M^2 \sum_{k=1}^{\infty} |v_k|^2= M^2 \|v\|^2,
\end{equation*}
where we define $M:=\max\{M_1, M_2\}$. Hence $f$ is continuous and $\|f\|\leq M$. Now, prove the continuity of the inverse $f^{-1}$ is equivalent to prove that $\|v\|\leq \|f(v)\|$ for all $v\in\A(G)$. By the previous calculations
\begin{equation}\nonumber
\|v\|^2 =\sum_{k\in V_1} |v_k|^2 + \sum_{k\in V_2} |v_k|^2 \leq M_1^2\sum_{k\in V_1} |v_k|^2 + M_2^2\sum_{k\in V_2} |v_k|^2=\|f(v)\|^2,
\end{equation}
where we use that $M_1,M_2\geq 1$. Thus $f$ is an isomorphism of Hilbert evolution algebras.
 \end{proof}

Proposition \ref{theo:generalization} proves the first affirmation of Theorem \ref{theo:criterio}. Note that this result  there is no restriction regarding $\ker A$ so it covers graphs for which $\ker A \neq \{0\}$ as well. For proving the other affirmation of Theorem \ref{theo:criterio} we appeal to the following results which extend the ones obtained in \cite{PMPT}.

If $\A$ is an Hilbert evolution algebra  with orthonormal basis $\{e_i\}_{i\in \mathbb{N}}$ and $v=\sum_{i=1}^{\infty}v_ie_i \in \A$ we denote by $\Omega_v$ the set $$\Omega_v := \{i\in \mathbb{N}: v_i \not= 0\}. $$
And to simplify notation we write $\Omega_i$ for $\Omega_{e_i}$.

\begin{proposition}\label{prop:homo between HEA}
Let $G$ be a non-singular graph with uniformly bounded degree. Let $\A$ be a separable Hilbert space, $\{e_i\}_{i\in \mathbb{N}}$ an orthonormal basis, $\A(G)$ and $\A_{\RW}(G)$ be the associated Hilbert evolution algebras. If $f:\A_{\RW}(G)\longrightarrow \A(G)$ is an  homomorphism of Hilbert evolution algebras then either f is the null map or satisfies the following,
\begin{enumerate}[label=(\roman*)]
\item \label{injectivity} $f$ is injective. 
\item  \label{partition} The family $\{\Omega_{i}\}_{i\in \mathbb{N}}$ is a partition of $\mathbb{N}$. 
\end{enumerate}

\end{proposition}

\begin{proof} Suppose that $f$ is not null. Let us write
\begin{equation}\label{eq:f(ei)}
f(e_i)=\sum_{k=1}^{\infty} t_{ki}e_k, \quad \text{ for all } i\in \mathbb{N},
\end{equation}
Note that the scalars $t_{ik}$ must satisfy $\sum_{k=1}^\infty |t_{ki}|^2=\|f(e_i)\|^2<\infty$. 
The condition $f(e_i)\cdot f(e_j)=0$, for any $i\neq j$, implies
\begin{equation}\label{eq:f(ei)f(ej)=0}
0=\sum_{k=1}^\infty t_{ki} t_{kj} e_k^2 =\sum_{k=1}^\infty t_{ki} t_{kj} \left(\sum_{r=1}^\infty  a_{rk} e_r\right)=\sum_{r=1}^\infty \left(\sum_{k=1}^\infty t_{ki} t_{kj} a_{rk}\right) e_r.
\end{equation}
For every $i,j\in \mathbb N$, $i\neq j$ we define $w_{ij}:=\sum_{k=1}^{\infty}t_{ki}t_{kj}e_k$. We affirm that $w_{ij} \in \A$. In fact,
\begin{equation} \nonumber
\sum_{k=1}^{\infty}|t_{ki}t_{kj}|^2  =\sum_{k=1}^{\infty}|t_{ki}|^2|t_{kj}|^2 
\leq \|f(e_j)\|^2 \sum_{k=1}^{\infty}|t_{ki}|^2 
=\|f(e_j)\|^2\|f(e_i)\|^2<\infty,
\end{equation}
where we use $|t_{kj}|^2\leq \sum_{k=1}^{\infty}|t_{kj}|^2 =\|f(e_j)\|^2$ for all $k,j\in\mathbb{N}$.
Then, we can  apply the evolution operator, 
\begin{equation} \nonumber
C(w_{ij})=\sum_{r=1}^\infty \left(\sum_{k=1}^\infty t_{ki} t_{kj} a_{rk}\right) e_r=0,
\end{equation}
where we use \eqref{eq:f(ei)f(ej)=0}. 
Now, we use \eqref{eq:evol vs adja} to obtain
\begin{equation}\nonumber
C(w_{ij})=C(U(v_{ij}))=U(A(v_{ij}))=0,
\end{equation}
where $v_{ij}:=U^{-1}(w_{ij})=\sum_{k=1}^{\infty}t_{ki}t_{kj}\delta_k\in \ell^2(\mathbb{N})$. But $U$ is an isomorphism, thus $A(v_{ij})=0$, for any $i\neq j$. This implies that $v_{ij}=0$, for any $i\neq j$, because the graph is non-singular, i.e., $\ker A=\{0\}$. Hence
\begin{equation}\label{eq:tij}
t_{ki}t_{kj}=0, \text{ for any }i,j,k\in V \text{ with }i\neq j.
\end{equation}
Thus, for $k\in \mathbb{N}$ we have $t_{ki} =0$ for all $ i\in \mathbb{N}$, or there exists at most one $i:=i(k)\in \mathbb{N}$ such that $t_{ki}  \not = 0$ and  $t_{kj} =0$ for all $j  \not = i$. Since $\Omega_{i}=\{k\in \mathbb{N}: t_{ki}\neq 0\}$, we have that
\begin{equation}\label{eq:supp}
\Omega_{i} \cap \Omega_{j}=\emptyset, \text{ for }i\neq j.
\end{equation}
To prove \ref{injectivity} we consider two cases.\\
\noindent
{\bf Case 1.} Suppose there exist $i\in \mathbb{N}$ such that $t_{ki}=0$ for all $k\in V$. Then, by \eqref{eq:f(ei)} we have $f(e_i)=0$ implying
\begin{equation}\label{eq:fei2=0}
0=f(e_i^2) =f\left(\sum_{\ell=1}^\infty \frac{a_{i\ell}}{\deg (i)} e_{\ell}\right)=\sum_{\ell=1}^\infty \frac{a_{i\ell}}{\deg(i)} f(e_{\ell}).
\end{equation}
By \eqref{eq:supp} we conclude that $f(e_{\ell})=0$ for any $\ell$ such that $a_{i\ell}=1$. Indeed, if $f(e_\ell)\neq 0$, then $t_{kl}\neq 0$ for some $k$, i.e., $k\in\Omega_\ell$ and \eqref{eq:fei2=0} says that there must be some $j\neq \ell$ with $t_{kj}\neq 0$, i.e., $k\in\Omega_j$ which is not possible by \eqref{eq:supp}. Then, for any $\ell \in \mathcal{N}(i)$ it holds that $f(e_\ell)=0$. This procedure can be repeated for any $\ell \in \mathcal{N}(i)$, i.e., we have that $f(e_v)=0$ for any $v\in\mathcal{N}({\ell})$. As we are dealing with a connected graph, this procedure can be iterated a finite number of times until to cover all the vertices of $G$. Therefore we can conclude that $f(e_{i})=0$ for any $i\in \mathbb{N}$, a contradiction, because $f$ is not null.

\noindent {\bf Case 2.} For all $i\in \mathbb{N}$ there exist $k\in \mathbb{N}$ such that $t_{ki}\neq 0$. Then, by \eqref{eq:f(ei)} we have that $f(e_i)\neq 0$ for all $i\in \mathbb{N}$ as desired. 
In this case we can prove that $f$ is injective.
Suppose that $v\in\ker f$ and $v\neq 0$ then, we have $v=\sum_{i=1}^{\infty} v_ie_i$ with some $v_j\neq 0$.
For this $j\in \mathbb{N}$ there exist at most one $\ell\in \mathbb{N}$ such that $t_{\ell j}\neq 0$.
Now, $f(v)=0$ implies 
\begin{equation}\nonumber
\sum_{i=1}^{\infty}v_if(e_i)= \sum_{k=1}^{\infty}\Big(\sum_{i=1}^{\infty}v_it_{ki}\Big )e_k=0,
\end{equation}
that is 
\begin{equation}\nonumber
\sum_{i=1}^{\infty} v_i t_{ki}=0 ,\,\,\,\text{for all } k\in \mathbb{N}.
\end{equation}
In particular $\sum_{i=1}^{\infty} v_i t_{\ell i}=0$ and by \eqref{eq:tij} we can conclude $v_jt_{\ell j}=0$, which is an absurd because both terms are different from $0$.

For \ref{partition}, let us prove that
\begin{equation*}
\displaystyle \bigcup_{i\in \mathbb{N}} \Omega_{i}= \mathbb{N}.
\end{equation*}
Let $i\in \mathbb{N}$. We want to prove that $i\in \Omega_{\ell}$, that is, $t_{il}\neq 0$, for some $\ell \in \mathbb{N}$. 
To fix notation let us rewrite \eqref{eq:f(ei)} as 
\begin{equation*}\label{eq:f(ei) 1}
f(e_s) = \sum_{h\in\Omega_s} t_{hs} e_h,\,\,\, \text{for any } s\in \mathbb{N}.
\end{equation*}
In particular $f(e_i) = \sum_{h\in\Omega_i} t_{hi} e_h\neq 0$ and we can choose some $j\in \Omega_i$.
Now, as $G$ is connected there is a finite path  $\{j=j_0,j_1,\ldots, j_n,j_{n+1}=i\}$ joining $j$ to $i$. We want to prove that every index in this set belongs to some $\Omega_{i_k}$. Note that $j_1\in \N(j)\subset \N(\Omega_i)$, which implies
\begin{equation*}
f(e_i) \cdot f(e_i)=\sum_{h\in\Omega_i} t_{hi}^2 e_h^2 = \sum_{s\in \mathcal{N}(\Omega_i)}\beta_{s} e_{s} =\,\,\beta_{j_1}e_{j_1}+ \sum_{s\in \mathcal{N}(\Omega_i)\setminus\{j_1\}}\beta_{s} e_{s}.
\end{equation*}
But $f(e_i) \cdot f(e_i)=f(e_i^2)$, hence
\begin{equation*}
f(e_i^2)=\frac{1}{\deg(i)} \,\sum_{\ell\in \mathcal{N}(i)} f(e_\ell)=\beta_{j_1}e_{j_1}+ \sum_{s\in \mathcal{N}(\Omega_i)\setminus\{j_1\}}\beta_{s} e_{s}.
\end{equation*}
Which implies, by \eqref{eq:supp}, that there exist some $i_1\in \N(i)$ such that $j_1\in\Omega_{i_1}$.
Repeating the process $n$ times we obtain  $i_n\in \N(i_{n-1})$ such that $j_n\in\Omega_{i_n}$.
But $i\in\N(j_n)\subset\N(\Omega_{i_n})$, then
\begin{equation*}
\nonumber
f(e_{i_n}) \cdot f(e_{i_n})= \sum_{s\in \mathcal{N}(\Omega_{i_{n}})} \beta_{s} \, e_{s} =\beta_{i}e_{i}+ \sum_{s\in \mathcal{N}(\Omega_{i_{n}})\setminus \{i\}} \beta_{s} \, e_{s}.
\end{equation*}
And therefore 
$$
f(e_{i_{n}}^2)=\frac{1}{\deg(i_n)} \,\sum_{\ell\in \mathcal{N}(i_n)} f(e_\ell)=\beta_{i}e_{i}+ \sum_{s\in \mathcal{N}(\Omega_{i_{n}})\setminus \{i\}} \beta_{s} \, e_{s},
$$
which implies that there exist some $\ell\in \N(i_n)$ such that
$i\in\Omega_{\ell}$, i.e.,
\begin{equation*}
f(e_{\ell})=t_{i\ell}e_{i}+\ldots \quad, \text{ with }t_{i\ell}\neq 0.
\end{equation*}
Note that if $i\in\Omega_l$, $i\in\Omega_k$, for $k\neq l$, we must have $t_{il}\neq0$ and $t_{ik}\neq 0$. But this is not possible by \eqref{eq:tij}. Thus the $\Omega_i$'s are a partition of $\mathbb{N}$.


\end{proof}

In the previous result we find some properties of the homomorphisms of Hilbert evolution algebras between $\A_{\RW}(G)$ and $\A(G)$. If we assume that $f$ is an isomorphism we have a stronger result.
\begin{corollary}\label{corol:isomorphism of HEA}
In the conditions of Proposition \ref{prop:homo between HEA} suppose further that $f:\A_{\RW}(G)\longrightarrow \A(G)$ is an isomorphism of Hilbert evolution algebras. Then 
\begin{equation}\label{eq:f isomorphism}
f(e_i) = \alpha_i e_{\pi(i)}
\end{equation}
where  $\alpha_i \neq 0$, for all $i\in \mathbb N$ and $\pi :\mathbb N\rightarrow \mathbb N$ is a bijection.
Moreover, there is $M>0$ such that $|\alpha_i|\leq M$, for all $i\in \mathbb{N}$.
\end{corollary}
\begin{proof}
We know by the proof of Proposition \ref{prop:homo between HEA} that  if $f$ is given by the expression \eqref{eq:f(ei)} then the scalars $\{t_{ki}\}_{i,k\in\mathbb{N}}$ satisfy \eqref{eq:tij}. Since  $f(e_i)\neq 0$ for all $i\in \mathbb{N}$ we have that $\Omega_{i}\not= \emptyset$. Moreover, we affirm that $\Omega_{i}$  are unitary sets, for all $i$. In this case, the application
$ \pi:\mathbb{N}\longrightarrow \mathbb{N}$ defined by  $\pi(i)=k$, where $k\in \Omega_{i}$ is well defined and, by Proposition \ref{prop:homo between HEA}, $\pi$ is a bijection. As a consequence of this, we have that
\begin{equation}
\nonumber
f(e_i)=\sum_{k=1}^{\infty} t_{ki}e_k= t_{\pi(i)i}e_{\pi(i)}=\alpha_i e_{\pi(i)},
\end{equation}
where we define $\alpha_i:=t_{ki}$. 

Let us see that $\Omega_{j}$ is an unitary set, for all $j\in \mathbb{N}$. Suppose that for some $i\in \mathbb{N}$ there are $k, \ell \in \Omega_{i}$ and $k\neq \ell$.  That means that $t_{ki}\not =0$ and $t_{\ell i}\not =0$. Since  $f$ is  surjective, there exist  $v^k\in \A_{\RW}(G)$ such that $f(v^k)=e_k$. If $v^k=\sum_{j=1}^{\infty}v^k_j e_j$ then
\begin{equation}
\nonumber e_k=f(v^k)=\sum_{j=1}^{\infty} v^k_jf(e_j) =\sum_{m=1}^{\infty}\left (\sum_{j=1}^{\infty}v^k_j t_{kj}\right)e_m.
\end{equation}
Therefore
\begin{equation*}\label{eq:f sobre}
\sum_{j=1}^{\infty} v^k_j t_{mj}=\delta_{mk},\;\;\;\text{ for all }m\in \mathbb{N}.
\end{equation*}
If $m=k$, we have $\sum_{j=1}^{\infty} v^k_j t_{kj}=1$. Since $t_{ki}\neq 0$, by  \eqref{eq:tij},  we have that $t_{kj}=0$ for $j\neq i$. Then $v^k_i t_{ki}=1$, i.e.,
\begin{equation}\label{eq:v^k_i}
v^k_i=t_{ki}^{-1}.
\end{equation}
If $m= \ell$, we have  $\sum_{j=1}^{\infty} v^k_j t_{\ell j}=0$. As before $t_{\ell i}\neq 0$ implies $t_{\ell j}=0$ for $j\neq i$, giving $v^k_it_{\ell i}=0$. That is $v^k_i=0$, which is a contradiction of \eqref{eq:v^k_i}. Therefore $k=\ell$, as desired.



The continuity of $f$ implies that there exist  $M>0$ such that
\begin{equation*}
\|f(e_i)\| =|\alpha_i|\|e_{\pi(i)}\|=|\alpha_i|\leq M, \quad\text{ for all } i\in \mathbb{N}.
\end{equation*}
\end{proof}

\begin{proposition}\label{theo:sufficient}
Let $G$ be a graph. Assume that there exist an Hilbert evolution isomorphism $f:\mathcal{A}(G)\longrightarrow \mathcal{A}_{\RW}(G)$ defined by 
\begin{equation}\label{eq:isofunc}
f(e_i) = \alpha_i e_{\pi(i)},\;\;\;\text{ for all }i\in \mathbb{N},
\end{equation}
where $\alpha_i \neq 0$ is a scalar and  $\pi :\mathbb N\rightarrow \mathbb N$ is a bijection. Then $G$ is a biregular graph or a regular graph. 
\end{proposition}

\begin{proof} The proof of this result is the same than the one of \cite[Proposition 2.9]{PMPT} so we only sketch it and we refer the reader to \cite{PMPT} for details. If we assume that  $f:\mathcal{A}(G)\longrightarrow \mathcal{A}_{\RW}(G)$ defined by \eqref{eq:isofunc} is an Hilbert evolution isomorphism then, by $f(e_{i}^2)=f(e_i)\cdot f(e_i)$ for any $i\in \mathbb{N}$ we can conclude that $ \deg(i)=\deg(\pi(i))$ for all $i\in \mathbb{N}$. From here, it follows $\alpha_\ell = \alpha_{i}^2/\deg(i)$, for all $\ell \in \mathcal{N}(i)$. After some algebraic manipulations, and appealing to the symmetry of the previous relations (note that if $\ell \in \mathcal{N}(i)$ then  $i \in \mathcal{N}(\ell)$) we get for $\ell_1, \ell_2 \in \mathcal{N}(i)$
$$
\frac{\alpha_{\ell_{1}}^2}{\deg(\ell_{1})} = \alpha_i = \frac{\alpha_{\ell_{2}}^2}{\deg(\ell_{2})}.$$
As a consequence, we obtain the following condition on the degrees in the graph: for any $i\in V,$ if $\ell_1, \ell_2 \in \mathcal{N}(i)$  then $\deg(\ell_1)=\deg(\ell_2)$. This in turns implies that the graph must to be or a regular or a biregular graph, independently of the number of vertices being finite or infinite.
\end{proof}

\subsection{Proof of the Theorem \ref{theo:criterio}}\label{sec:proofs}
The first affirmation is the Proposition \ref{theo:generalization}.
Now, suppose that $G$ is non-singular and $\A_{\RW}(G)\cong \A(G)$, then by Corollary \ref{corol:isomorphism of HEA} we have that the isomorphism $f$ has the form given in the equation \eqref{eq:f isomorphism}. Thus, we are in the conditions of Proposition \ref{theo:sufficient}, i.e., $G$ is regular or biregular, as desired.

\section{Isometric Isomorphisms}
Note that the isomorphism of the Corollary \ref{corol:isomorphism of HEA} is not unitary in general. This is because $\{e_i\}_{i\in\mathbb{N}}$ is a orthonormal natural basis of both $\A(G)$ and $\A_{\RW}(G)$, thus
\begin{equation*}
\langle f(e_i),f(e_j)\rangle=\alpha_i\alpha_j \langle e_{\pi(i)}, e_{\pi(j)}\rangle=\alpha_i\alpha_j\delta_{\pi(i)\pi(j)} 
\end{equation*}
which in general is not equal to $\langle e_i,e_j \rangle=\delta_{ij}$.
If we want an isometric isomorphism, the problem is that we can find an isomorphism of evolution algebras which is not unitary and there are unitary operators between the separable Hilbert spaces $\A(G)$ and $\A_{\RW}(G)$ which are not homomorphisms of evolution algebras. In other words, respect both structures at the same time is too restrictive. But the equation \eqref{eq:f isomorphism} gives us a hint to define a new inner product on $\A_{\RW}(G)$ to obtain an isometric isomorphism, i.e., a homomorphism which is a unitary operator.

Let $f:\A_{\RW}(G)\longrightarrow \A(G)$ be an isomorphism of Hilbert evolution algebras defined by equation \eqref{eq:f isomorphism} with $\alpha_i \neq 0$, $i\in \mathbb N$. With this we can define a new basis on $\A_{\RW}(G)$ given by
\begin{equation*}
\widetilde{e_i}:=\alpha_i^{-1}e_i, \quad\quad\text{ for all } i\in\mathbb{N}.
\end{equation*}
Thus we have and induced inner product on $\A_{\RW}(G)$, turning the basis $\{\widetilde{e_i}\}_{i\in\mathbb{N}}$ into an orthonormal basis; that is
\begin{equation*}
\langle \widetilde{e_i},\widetilde{e_j}\rangle_f:=\delta_{ij}
\end{equation*}
and we define the new Hilbert space $\A^f=(\A,\langle\cdot,\cdot\rangle_f)$ with this structure. Note that in this new product the basis $\{e_i\}_{i\in\mathbb{N}}$ is no longer orthonormal, in fact  
\begin{equation*}
\langle e_i,e_j\rangle_f=\alpha_i\alpha_j\delta_{ij} .
\end{equation*}
This Hilbert space $\A^f$ induces the Hilbert evolution algebra  $\A^f_{\RW}(G)$, where the product is given by
\begin{equation}
\nonumber 
\widetilde{e}_i \cdot \widetilde{e}_i =\displaystyle \sum_{k=1}^{\infty}\frac{a_{ki}}{\deg(i)} \widetilde{e}_k \quad  \text{ and }\quad \widetilde{e}_i \cdot \widetilde{e}_j=0, \text{ if }i\neq j.
\end{equation}
\begin{proposition}\label{prop:induced unitary isomorphism}
Let $f:\A_{\RW}(G)\longrightarrow \A(G)$ be an isomorphism of Hilbert evolution algebras defined by \eqref{eq:f isomorphism}, where $\alpha_i \neq 0$, $i\in \mathbb N$. Then the induced Hilbert evolution algebra $\A^f_{\RW}(G)$ is unitarily isomorphic to $\A(G)$.
\end{proposition}
\begin{proof}
We want to prove that there exist an unitary Hilbert algebra isomorphism
\begin{equation*}
\widetilde{f}:\A^f_{\RW}(G)\longrightarrow \A(G)
\end{equation*}
Define the unitary isomorphism $\varphi:\A^f_{\RW}(G)\longrightarrow\A_{\RW}(G)$ taking orthonormal basis into orthonormal basis, i.e.,
\begin{equation*}
\varphi(\widetilde{e}_i):=e_i.
\end{equation*}
Then we can define the new isomorphism by
\begin{equation*}
\widetilde{f}:=f\circ\varphi
\end{equation*}
that is $\widetilde{f}(\widetilde{e}_i)=e_{\pi(i)}$ for all $i\in\mathbb{N}$.
This $\widetilde{f}$ is linear, bijective, continuous and with continuous inverse, because is composition of maps with these properties.
It remains to prove that $\widetilde{f}$ is unitary, but
\begin{equation*}
\langle \widetilde{f}(\widetilde{e}_i),\widetilde{f}(\widetilde{e}_j)\rangle= \langle e_{\pi(i)},e_{\pi(j)}\rangle =\delta_{\pi(i)\pi(j)} =\delta_{ij}.
\end{equation*}
Thus $\widetilde{f}$ takes orthonormal basis into orthonormal basis, which, by linearity, implies
\begin{equation*}
\langle \widetilde{f}(v),\widetilde{f}(w)\rangle= \langle v,w\rangle_f,
\end{equation*}
for all $v,w\in \A^f_{\RW}(G)$, as desired.
\end{proof}

Now, we can formulate a similar result to Theorem \ref{theo:criterio}.
\begin{theorem}\label{theo:criterio unitario}
Let $G$ be a graph with uniformly bounded degree. If $G$ is a regular or a biregular graph, then $\A^f_{\RW}(G)\cong_U \A(G)$, for some $f$ satisfying \eqref{eq:f isomorphism}. Moreover, if $G$ is non-singular and $\A^f_{\RW}(G)\cong_U \A(G)$ for some $f$ satisfying \eqref{eq:f isomorphism}, then $G$ is a regular or a biregular graph. 
\end{theorem}
\begin{proof}
If $G$ is a regular or a biregular graph, then by Proposition \ref{theo:generalization} we know that there exist an isomorphism of the form
\begin{equation*}
f(e_i) = \alpha_i e_{i},\;\;\;\text{ for all }i\in \mathbb{N},
\end{equation*}
where $\alpha_i\neq 0$. Thus, we are in the conditions of Proposition \ref{prop:induced unitary isomorphism}, that is, $\A^f_{\RW}(G)$ is unitarily isomorphic to $\A(G)$. On ther hand, suppose that $G$ is non-singular and $\A^f_{\RW}(G)$ is unitarily isomorphic to $\A(G)$, for some $f$ satisfying \eqref{eq:f isomorphism}. In particular, we are in the conditions of Proposition \ref{theo:sufficient}. It follows that $G$ is a biregular graph or a regular graph. 
\end{proof}

\section{Discussion}

\subsection{Homomorphisms that are not isomorphisms}
In the finite case it was proved \cite[Proposition 2.14]{PMPT} that, if $G$ is a non-singular graph and $\A_{\RW}(G)$ and $\A(G)$ are not isomorphic, the only homomorphism of evolution algebras is the null map. 
In our case we prove that, for non-singular graphs, the homomorphisms are always injective (see Proposition \ref{prop:homo between HEA}), but it was not possible to us, to prove that they are surjective (in the finite case this is a direct consequence of injectivity). Thus, this remains as an open question, that is: there exist an injective but not surjective homomorphism between $\A_{\RW}(G)$ and $\A(G)$ for a non-singular graph?

On the other hand, it is worth pointing out that for some singular graphs we can find homomorphisms different of the isomorphisms identified in Corollary \ref{corol:isomorphism of HEA}. Let us explain this statement by considering the twin partition of a given graph. Let us start with some additional definitions from Graph Theory. We say that vertices $i$ and $j$ of $G$ are twins if they have exactly the same set of neighbors, i.e. $\mathcal{N}(i)=\mathcal{N}(j)$. We notice that by defining the relation $\sim_{t}$ on the set of vertices $V$ by $i\sim_{t} j$ whether $i$ and $j$ are twins, then $\sim_{t}$ is an equivalence relation. An equivalence class of the twin relation is referred to as a {\it twin class}. In other words, the twin class of a vertex $i$ is the set $\{j\in V:i \sim_{t} j\}$. The set of all twin classes of $G$ is denoted by $\Pi(G)$ and it is referred to as the {\it twin partition} of $G$. A graph is {\it twin-free} if it has no twins. 
We denote by $G/\Pi$ the quotient graph of $G$ with respect to $\Pi$; which is the graph obtained by merging each class of $\Pi$ into a single vertex. In this case, two vertices in $G/\Pi$ are neighbors if and only if their respective twin classes in $G$ are neighbors in the sense that any vertex of one of the classes is a neighbor of any vertex of the other one. Notice that this operation on the graph may be seen as the one called twin reduction in literature, see for instance \cite[Section 3.3]{lovasz/2012}. It is not difficult to see that this construction leads to a uniquely determined twin-free graph.

\begin{proposition}\label{prop:twin construction}
Let $G$ be a graph with uniformly bounded degree. Let $\Pi(G)$ be the twin partition of $G$ and let $G/\Pi$ be the quotient graph of $G$ with respect to $\Pi$. 
If $G/\Pi$ is a regular or a biregular graph, then there exists a non-injective, surjective homomorphism between $\A_{\RW}(G)$ and $\A(G)$.
\end{proposition}

\begin{proof}
Let $G$ be a graph with uniformly bounded degree, and let $\A$ be a separable Hilbert space with an orthonormal basis $\{e_i\}_{i\in\mathbb{N}}$. Consider the Hilbert evolution algebras $\A(G)$ and $\A_{\RW}(G)$ defined, respectively, by 
$$
e_i \cdot e_i =\displaystyle \sum_{k=1}^{\infty} a_{ki} e_k \quad \text{ and }\quad e_i \cdot e_j=0, \text{ if }i\neq j,
$$
see Proposition \ref{prop:HEA of uniformily bounded degree}, and by 
$$
e_i \cdot e_i =\displaystyle \sum_{k=1}^{\infty} \frac{a_{ki}}{\deg(i)} e_k \quad  \text{ and }\quad e_i \cdot e_j=0, \text{ if }i\neq j,
$$
see Proposition \ref{prop:SRWhea}. Now, let $\Pi(G)=\{[i_k]\}_{k=1}^{|\Pi(G)|}$ be the twin partition of $G$, where $[i_k]$ denotes the twin class of vertex $i_k$. This induces the closed linear span given by 
\begin{equation*}
\A^\Pi:=\overline{\mathrm{span}}\{e_{i_k}:k=1,\ldots,|\Pi(G)|\}
\end{equation*}
which is a closed linear subspace of $\A$. Now consider the quotient graph of $G$ with respect to $\Pi$, $G/\Pi$ and note that this graph is of uniformly bounded degree. Thus, we have the induced Hilbert evolution algebras $\A^\Pi(G/\Pi)$, $\A^\Pi_{\RW}(G/\Pi)$. If the quotient graph $G/\Pi$ is a regular or a biregular graph, then Theorem \ref{theo:criterio} guarantees the existence of an isomorphism $f:\A^\Pi_{\RW}(G/\Pi)\longrightarrow \A^\Pi(G/\Pi)$. This isomorphism allow us to construct an homomorphism denoted by the same letter, $f:\A_{\RW}(G)\longrightarrow \A(G)$ and defined by
\begin{equation*}
f(e_i):=f(e_{i_k}) \quad,\quad\text{for all } i\in\mathbb{N}\,,\, i \sim_{t} i_k.
\end{equation*}
Note that $f$ is not injective because two different generators with twin indices have the same image. And $f:\A_{\RW}(G)\longrightarrow \A(G)$ is surjective because $f:\A^\Pi_{\RW}(G/\Pi)\longrightarrow \A^\Pi(G/\Pi)$ is surjective.
\end{proof}
Thus, the Proposition \ref{prop:twin construction} gives examples of singular graphs where we have a non-injective, surjective homomorphism between $\A_{\RW}(G)$ and $\A(G)$.

\subsection{On the condition of the uniform boundedness of the graph}

Since we are dealing with graphs with infinitely many vertices we appeal to the self-adjointness of the adjacency operator, defined by \eqref{eq:adjacency operator02}, to guarantee that it is well defined into the entire space $\ell^2(\mathbb{N})$. Thus Theorem \ref{theo:adjacency oper bounded} lead us to restrict our attention to the wide class of graphs with uniformly bounded degree. Moreover, Proposition \ref{prop:HEA of uniformily bounded degree} implies that only for these graphs we can define a Hilbert evolution algebra. In order to find examples of graphs with, and without, uniformly bounded degree one can appeal to the spherically symmetric trees. A rooted tree $T$, with root ${\bf 0}$, is called spherically symmetric if for any vertex $v$ its degree $\deg(v)$ depends only on $\dist({\bf 0},v)$. See \cite[Chapter 3, Section 2]{lyons} for some properties. The $2$ periodic tree with degrees $2$ and $3$, $\mathbb{T}_{2,3}$, is an example of infinite graph with uniformly bounded degree, see Figure \ref{fig:tree}(a). Then Proposition \ref{prop:HEA of uniformily bounded degree} guarantees the existence of a well defined Hilbert evolution algebra $\A(\mathbb{T}_{2,3})$. On the other hand, we cannot claim the same for the factorial tree $\mathbb{T}_{!}$, which is an example of infinite graph which is locally finite but it has not uniformly  bounded degree. Indeed, $\mathbb{T}_{!}$ is the tree such that any vertex a distance $n$ from the root has degree $n+2$ so the number of vertices at distance $n$ from the root is $n!$, see Figure \ref{fig:tree}(b).

\begin{figure}
    \centering
    \subfigure[][The $2$ periodic tree with degrees $2$ and $3$, $\mathbb{T}_{2,3}$.]{

\begin{tikzpicture}[scale=0.8]
\node at (0,-4.8) {};
\draw (-1.5,0)--(1.5,0);
\draw (-1.5,0)--(-3.5,2);
\draw (-1.5,0)--(-3.5,-2);
\draw (1.5,0)--(3.5,2);
\draw (1.5,0)--(3.5,-2);

\draw [dotted] (3.5,2)--(3.5,3);
\draw [dotted] (3.5,2)--(4.5,2);
\draw [dotted] (3.5,-2)--(3.5,-3);
\draw [dotted] (3.5,-2)--(4.5,-2);
\draw [dotted] (-3.5,2)--(-3.5,3);
\draw [dotted] (-3.5,2)--(-4.5,2);
\draw [dotted] (-3.5,-2)--(-3.5,-3);
\draw [dotted] (-3.5,-2)--(-4.5,-2);

\filldraw [black] (0,0) circle (1.5pt);
\filldraw [black] (1.5,0) circle (1.5pt);
\filldraw [black] (-1.5,0) circle (1.5pt);
\filldraw [black] (2.5,1) circle (1.5pt);
\filldraw [black] (2.5,-1) circle (1.5pt);
\filldraw [black] (-2.5,1) circle (1.5pt);
\filldraw [black] (-2.5,-1) circle (1.5pt);

\filldraw [black] (3.5,2) circle (1.5pt);
\filldraw [black] (3.5,-2) circle (1.5pt);
\filldraw [black] (-3.5,2) circle (1.5pt);
\filldraw [black] (-3.5,-2) circle (1.5pt);

\filldraw [white] (0,-4.2) circle (1.5pt);

\end{tikzpicture}
}\qquad\subfigure[][The factorial tree $\mathbb{T}_{!}$]{

\tikzstyle{level 1}=[sibling angle=120]
\tikzstyle{level 2}=[sibling angle=130]
\tikzstyle{level 3}=[sibling angle=60]
\tikzstyle{level 4}=[sibling angle=30]

\tikzstyle{edge from parent}=[segment length=0.8cm,segment angle=10,draw]
\begin{tikzpicture}[scale=0.8,rotate=0,grow cyclic,shape=circle,minimum size = 1.5pt,inner sep=1pt,level distance=15mm,
                    cap=round]
\node[fill] {} child [\A] foreach \A in {black}
    { node[fill] {} child [color=\A!50!\B] foreach \B in {black,black}
        { node[fill] {} child [color=\A!50!\B!50!\C] foreach \C in {black,black,black}
            { node[fill] {} child [color=\A!50!\B!50!\C!50!\D] foreach \D in {black,black,black,black}
            { node[fill] {} }
        }
    }
    };
    
\node at (-0.4,-0.2) {\tiny ${\bf 0}$};
\node at (0,6) {};

\draw[dotted] (5.5,1.7) -- (5.9,1.77);
\draw[dotted] (5.5,-1.7) -- (5.9,-1.77);
\draw[dotted] (0.3,4) -- (0,4.4);
\draw[dotted] (0.3,-4) -- (0,-4.4);
\draw[dotted] (3.5,4.3) -- (3.7,4.8);
\draw[dotted] (3.5,-4.3) -- (3.7,-4.8);
\end{tikzpicture}

}
\caption{Example of infinite graphs with uniformly bounded degree (a), and without uniformly bounded degree (b).}\label{fig:tree}
\end{figure}
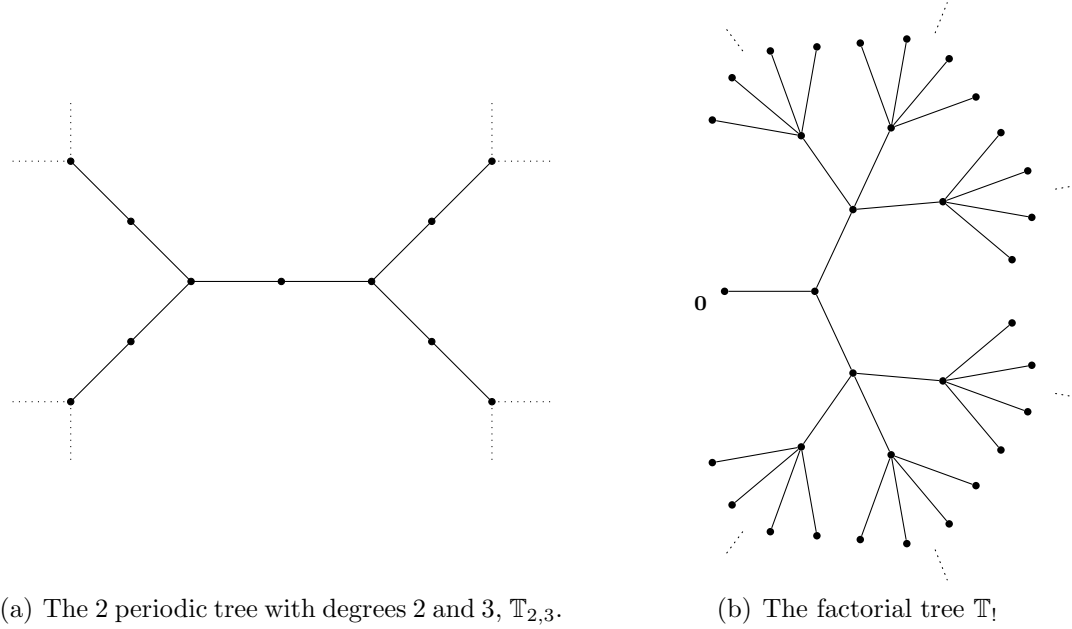

A direction of further research could be extend our framework to deal with graphs which are not of uniformly bounded degree. 
Then, we are interested in extend the theory to infinite graphs which are locally finite but not necessarily of uniformly bounded degree. Let $G$ be a locally finite graph, $\A$ a separable Hilbert space with an orthonormal basis $\{e_i\}_{i\in\mathbb{N}}$. We consider a different induced Hilbert space 
\begin{equation*}
\A_d=\biggl\{v=\sum_{i=1}^\infty v_ie_i\,:\, \sum_{i=1}^\infty |v_i|^2\deg(i)<\infty\biggr\},
\end{equation*}
where we define the weighted inner product,
\begin{equation*}
\langle v,w\rangle_d:=\sum_{i=1}^\infty v_i\overline{w_i}\deg(i),
\end{equation*}
for $v,w\in\A_d$, with corresponding norm $\|v\|_d:=(\sum_{i=1}^\infty |v_i|^2\deg(i))^{1/2}$. In this space the basis $\{e_i\}_{i\in\mathbb{N}}$ is no longer orthonormal, instead, we have $\|e_i\|_d^2=\deg(i)$. With this we have a result similar to Proposition \ref{prop:HEA of uniformily bounded degree}.

\begin{proposition}\label{prop:HEA of locally bounded degree}
Let $G$ be a graph with adjacency operator $A:D(A)\longrightarrow \ell^2(\mathbb{N})$ given by equation \eqref{eq:adjacency operator02} and $\A$ be a separable Hilbert space with an orthonormal basis $\{e_i\}_{i\in\mathbb{N}}$. If $G$ is a locally finite graph, the Hilbert space $\A_d$ induces a Hilbert evolution algebra $\A_d(G)$ defined by 
\begin{equation*}
e_i \cdot e_i =\displaystyle \sum_{k=1}^{\infty} \frac{a_{ki}}{\deg(k)^{1/2}} e_k \text{ and } e_i \cdot e_j=0, \text{ if }i\neq j,
\end{equation*}
\end{proposition}

\begin{proof}
Let $v\in\A_d(G)$. We must prove the continuity of the operator $L_v:\A_d(G)\longrightarrow\A_d(G)$. Write $v=\sum_{i=1}^\infty v_ie_i$, $w=\sum_{i=1}^\infty w_ie_i$, then
\begin{equation*}
\|L_v w\|^2_d=\sum_{k=1}^\infty\biggl|\sum_{i=1}^\infty v_iw_i \frac{a_{ki}}{\deg(k)^{1/2}}\biggr|^2\deg(k).
\end{equation*}
Hence
\begin{equation*}
\begin{array}{rl}
\|L_vw\|_d^2 
&\displaystyle\leq \sum_{k=1}^\infty\left(\sum_{i=1}^\infty |v_i|^2\right) \left(\sum_{i=1}^\infty |w_i\frac{a_{ki}}{\deg(k)^{1/2}}|^2\right) \deg(k)\\[3ex]
&\displaystyle = \left(\sum_{i=1}^\infty |v_i|^2\right) \sum_{k=1}^\infty \sum_{i=1}^\infty |w_i|^2 a_{ki}^2 \\[3ex]
&=\|v\|^2 \sum_{i=1}^\infty |w_i|^2\deg(i)\\[1.3ex]
&=\|v\|^2\|w\|_d^2
\end{array}
\end{equation*}
for all $w\in\A_d(G)$. It follows that $\|L_v\|_d\leq \|v\|$ for all $v\in\A_d(G)$.
\end{proof}

In this case, we can consider the construction at the beginning of Section \ref{sec:Hilbert evolution algebras associated to a graph} to obtain a well defined adjacency operator $A:D(A)\longrightarrow \ell^2(\mathbb{N})$, where the domain is given by \eqref{eq:domain of A}. Recall that, by construction, $A$ is a closed operator. Now the question is when $D(A)$ is dense in $\ell^2(\mathbb{N})$, and this depends on the structure constants $a_{ki}$. If $G$ is locally finite but not of uniformly bounded degree, $D(A)$ is a proper subset of $\ell^2(\mathbb{N})$. If not, we have that $D(A)=\ell^2(\mathbb{N})$ and $A$ is closed, which by the closed graph theorem implies that $A$ will be a bounded operator, which is not possible because is not of uniformly bounded degree (see Theorem \ref{theo:adjacency oper bounded}).
Besides, it is important to study conditions on the structure constants in order to have a well defined evolution operator $C:\A_d(G)\longrightarrow\A_d(G)$, which must be of the form, $C(e_i)=e_i^2=\sum_{k=1}^{\infty} (a_{ki}/\deg(k)^{1/2}) e_k$ (recall that uniformly bounded degree condition is longer used in this case). In this framework, the operators $A$ and $C$ will no longer be unitarily conjugated, thus the analysis is more complicated.


\section{Acknowledgements}

Part of this work was carried out during a visit of S.V. to the Federal University of Pernambuco. 

\bigskip

\end{document}